\documentclass[reqno,11pt]{amsart}
\usepackage{amsmath,amssymb,amsthm, verbatim}
\usepackage{url}
\usepackage[usenames, dvipsnames]{color}

\usepackage[letterpaper,hmargin=1in,vmargin=1in]{geometry}

\usepackage{graphicx}
\usepackage{enumerate,pinlabel}
\usepackage{mathrsfs,graphicx,url}
\usepackage[usenames,dvipsnames]{xcolor}
\usepackage[colorlinks=true,linkcolor=Black,citecolor=Black, urlcolor=Black]{hyperref}

\numberwithin{equation}{section}

\theoremstyle{plain}
\newtheorem{theorem}{Theorem}
\newtheorem*{theorem*}{Theorem}
\newtheorem*{lemma*}{Lemma}
\newtheorem{lemma}{Lemma}[section]

\newtheorem*{corollary*}{Corollary}

\theoremstyle{definition}

\theoremstyle{remark}

\newcommand{\R}{\mathbb{R}}
\newcommand{\US}{\mathbb{S}}

\newcommand\supp{\mathop{\rm supp}}

\newcommand\real{\mathop{\rm Re}}
\newcommand\imag{\mathop{\rm Im}}
\newcommand*{\defeq}{\mathrel{\vcenter{\baselineskip0.5ex \lineskiplimit0pt

                     \hbox{\scriptsize.}\hbox{\scriptsize.}}}%
                     =}

\definecolor{jeffColor}{RGB}{102, 0, 204}

\title{Semiclassical resolvent bounds for long range Lipschitz potentials}

\begin{document}

\author{Jeffrey Galkowski}
\address{Department of Mathematics, University College London, London, UK}
\email{j.galkowski@ucl.ac.uk}

\author{Jacob Shapiro}

\address{Department of Mathematics, University of Dayton, Dayton, OH 45469-2316}
\email{jshapiro1@udayton.edu}

\begin{abstract}
We give an elementary proof of weighted resolvent estimates for the semiclassical Schr\"odinger operator $-h^2 \Delta + V(x) - E$ in dimension $n \neq 2$, where $h, \, E > 0$. The potential is real-valued, $V$ and $\partial_r V$ exhibit long range decay at infinity, and may grow like a sufficiently small negative power of $r$ as $r \to 0$. The resolvent norm grows exponentially in $h^{-1}$, but near infinity it grows linearly. When $V$ is compactly supported, we obtain linear growth if the resolvent is multiplied by weights supported outside a ball of radius $CE^{-1/2}$ for some $C > 0$. This $E$-dependence is sharp and answers a question of Datchev and Jin. 
\end{abstract}
\maketitle 
\author
\section{Introduction and statement of results}

Let $\Delta \defeq \sum_{j=1}^n \partial^2_j \le 0$ be the Laplacian on $\mathbb{R}^n$, $n \neq 2$. Let $P$ denote the semiclassical Schr\"odinger operator
\begin{equation}
P = P(h) \defeq -h^2 \Delta + V(x) : L^2(\mathbb{R}^n) \to L^2(\mathbb{R}^n),\qquad x \in \R^n, \, h>0. \label{P}
\end{equation}
We use $(r, \theta) = (|x|, x/|x|) \in (0, \infty) \times \US^{n-1}$ for polar coordinates on $\R^n \setminus \{0\}$. Let $c_0, \, c_1 > 0$ and $0 \le \delta < \sqrt{8} -2$. Furthermore, let $p(r) > 0$ be bounded and decreasing to zero as $r \to \infty$, and suppose $0 < m(r) \le 1$ with 
\begin{equation} \label{defn m}
\lim_{r\to \infty}m(r)=0,\qquad (r+1)^{-1}m(r)\in L^1(0,\infty).
\end{equation}
We assume the potential $V : \R^n \to \R$, satisfies
\begin{gather}
V \in L^n(\R^n ; \R) + L^\infty(\R^n; \R), \label{L one plus L infty} \\
V \mathbf{1}_{|x| < 1}  \le c_1 r^{-\delta}, \label{V bound r small} \\
V \mathbf{1}_{|x| \ge 1} \le p(r)  \label{V bound r large}.
 \end{gather}
 In addition, we suppose there is function $\partial_r V \in L^1_{\text{loc}}(\R^n \setminus \{0\})$ such that, for each $\theta \in \US^{n-1}$, the function $(0, \infty) \ni r  \mapsto V(r, \theta) \defeq V(r \theta)$ has distributional derivative equal to $r \mapsto \partial_r V( r ,\theta)$, and
 \begin{gather}
\partial_r V( r, \theta) \mathbf{1}_{0< r  < 1}   \le c_1 r^{-1-\delta} \label{V prime bound r small}, \\
\partial_rV( r, \theta) \mathbf{1}_{r \ge 1} \le c_0r^{-1}m(r). \label{V prime bound r large}
\end{gather}

The prototypes we have in mind for \eqref{defn m} are the long range cases
\begin{equation} \label{long range m}
m = \log^{-1-\rho}(r + e), \qquad 
m = (r + 1)^{-\rho}, \qquad\qquad \rho>0.
\end{equation}

 When $n \ge 3$, \eqref{L one plus L infty} implies \eqref{P} is self-adjoint with respect to the domain $\mathcal{D}(P) = H^2(\R^n)$ \cite[Theorem 8]{ne64} . If $n =1$, then \eqref{P} is self-adjoint with respect to
\begin{equation*}
\mathcal{D}(P)=\{ u\in H^1(\mathbb{R})\mid u' \in L^\infty(\mathbb{R}),\,  Pu\in L^2(\mathbb{R})\},
\end{equation*}
see the proof in the Appendix. By \cite[Theorem 8]{ne64},  \eqref{P} is self adjoint with respect to $H^2(\R^n)$, provided $V \in L^p + L^\infty$ and $p \ge 2$, $p > n/2$, but for simplicity we work with \eqref{L one plus L infty}.

 For $E > 0$ and $s > 1/2$ fixed, and $h, \varepsilon > 0$, our goal is to establish $h$-dependent upper bounds on the weighted resolvent norms
 \begin{gather}
 g_s^{\pm}(h, \varepsilon) \defeq \| \langle x \rangle^{-s}  (P(h) - E \pm i \varepsilon)^{-1} \langle x \rangle^{-s}  \|_{L^2(\R^n) \to L^2(\R^n)}, \label{defn g} \\
g_s^{\pm}(h, M, \varepsilon) \defeq \| \langle x \rangle^{-s} \mathbf{1}_{|x| \ge M} (P(h) - E \pm i \varepsilon)^{-1} \mathbf{1}_{|x| \ge M } \langle x \rangle^{-s}  \|_{L^2(\R^n) \to L^2(\R^n)}. \label{defn g a}
 \end{gather}
Here, $\langle x\rangle=\langle r\rangle \defeq (1+r^2)^{1/2}$.
 
 In our Theorem, we bound ~\eqref{defn g} and~\eqref{defn g a} and show that, if $V$ is compactly supported, there are constants $C_1, \, h_0 > 0$, such that \eqref{defn g a} grows linearly in $h^{-1}$, provided $M \ge C_1 E^{-1/2}$, $\varepsilon > 0$ and $h \in (0, h_0]$. 
 
 \begin{theorem} \label{thm high dim} Fix $E > 0$ and $s > 1/2$. Suppose $V : \R^n \to \R$ satisfies \eqref{L one plus L infty} through \eqref{V prime bound r large}. There exist $M=M(E,p, c_0, c_1,\delta, m)$, $C_2= C_2(E,s, p , c_0,c_1,\delta, m)$, $C_3=C_3(E,s, p , c_0,c_1,\delta, m) > 0$ and $h_0 \in (0,1]$ so that, for all $\varepsilon >0$ and $h \in (0, h_0]$,
\begin{equation}
\label{e:totalEst high dim} \| \langle x \rangle^{-s}  (P(h) - E \pm i \varepsilon)^{-1} \langle x \rangle^{-s}  \|_{L^2(\R^n) \to L^2(\R^n)} \le  e^{C_3/h},
\end{equation}
 and
\begin{equation} \label{ext est high dim}
 \| \langle x \rangle^{-s} \mathbf{1}_{|x| \ge M} (P(h) - E \pm i \varepsilon)^{-1} \mathbf{1}_{|x| \ge M } \langle x \rangle^{-s}  \|_{L^2(\R^n) \to L^2(\R^n)} \le C_2 /h.
\end{equation}
Moreover, if $\supp V\subseteq B(0,R_0)$, then one can take $M = C_1(p,c_0,c_1,\delta,R_0)E^{-1/2}$.
 \end{theorem}
 
The main novelty of the Theorem is in the compactly supported case, where $M$ need not be larger than a constant times $E^{-1/2}$ . This seems to be the first general bound of the form $g_s^\pm(h, M , \varepsilon) \le Ch^{-1}$ for which $M$ depends explicitly on $E$. Moreover, owing to a construction of Datchev and Jin \cite[Theorem 1]{daji20}, this $E$-dependence of $M$ is optimal. In particular, if $V\in C_0^\infty(\mathbb{R}^n;\mathbb{R})$, $n\geq 2$ is radial and $\min(V)<0$, then there is $M\leq cE^{-1/2}$ with $g_s^{\pm}(h,M,\varepsilon)\geq e^{C/h}$. In addition, to the author's knowledge, this article is the first in this line of work to allow $V$ to be unbounded. We have included this to illustrate the flexibility of our methods, but do not expect the growth of $r^{-\sqrt{8} +2}$ near $r = 0$ is optimal.

Burq \cite{bu98} was the first to show $g^{\pm}_s \le e^{Ch^{-1}}$ for compactly supported perturbations of the Laplacian on $\R^n$. This bound was refined and extended many times \cite{vo00, bu02, sj02, cavo02, da14, sh19, vo20b} and is sharp in general, see \cite{ddz15}. Cardoso and Vodev \cite{cavo02}, refining Burq's earlier work \cite{bu02}, were the first to prove an exterior estimate of the form \eqref{ext est high dim}. They did so for smooth $V$ on a large class of infinite volume Riemannian manifolds. Exterior estimates were subsequently established under a wide range of regularity and geometric conditions \cite{ da14, vo14, rota15, dadeh16, sh19}.

Stronger bounds on $g_s^\pm$ are known when $V$ is smooth and conditions are imposed on the classical flow $\Phi(t) = \exp t(2 \xi \partial_x - \partial_x V(x) \partial_\xi)$ (note that $\Phi(t)$ may be undefined in our case). The key dynamical object is the \textit{trapped set} $\mathcal{K}(E)$ at energy $E > 0$, defined as the set of $(x, \xi) \in T^*\R^n$ such that $|\xi|^2 + V(x) = E$ and $|\Phi(t)(x, \xi)|$ is bounded as $|t| \to \infty$. If $\mathcal{K}(E) = \emptyset,$ that is, if $E$ is \textit{nontrapping}, Robert and Tamura \cite{rota87} showed $g^\pm_s \le Ch^{-1}$. We may think of \eqref{ext est high dim} as a low regularity analog; it says that applying cutoffs supported far away from zero removes the losses from \eqref{e:totalEst high dim} due to trapping.

Resolvent estimates such as \eqref{e:totalEst high dim} imply logarithmic local energy decay for the wave equation
\begin{equation} \label{wave equation}
\begin{cases}
(\partial_t^2 - c^2(x)\Delta) u(x,t) = 0, & (x,t) \in \left(\mathbb{R}^n \setminus \Omega \right) \times (0, \infty), \, n \ge 2, \\
 u(x,0) = u_0(x),\\
 \partial_t u(x,0) = u_1(x), \\
 u(t,x) = 0, & (x,t) \in  \partial \Omega \times (0,\infty),
\end{cases}
\end{equation}
where $\Omega$ is a compact (possibly empty) obstacle with smooth boundary, and the initial data are compactly supported. Such a decay rate was first proved by Burq \cite{bu98, bu02} for $c$ smooth. Logarithmic decay was subsequently established (when $\Omega = \emptyset$) for Lipschitz $c$ bounded from above and below \cite[Theorem 1]{sh18}. See also  \cite{be03, cavo04, bo11, mo16, ga19}. By leveraging \eqref{e:totalEst high dim}, we expect \cite[Theorem 1]{sh18}  extends to certain $c$ which tend to $0$ at a point.

 As shown in section XIII.7 of \cite{resi78}, the exterior bound \eqref{ext est high dim} is related to exterior smoothing and Strichartz estimates for Schr\"odinger propagators, see also \cite{botz07, mmt08} and Section 7.1 of \cite{dyzw19}. Furthermore, Christiansen \cite{ch17} used an estimate like \eqref{ext est high dim} to find a lower bound on the resonance counting function for compactly supported perturbations on the Laplacian on even-dimensional Riemannian manifolds. 

To prove the Theorem, we adapt the Carleman estimate from \cite{gash20}, which was used to prove a resolvent estimate for $L^\infty$ potentials. The key ingredients remain a weight $w(r)$ and phase $\varphi(r)$ that obey a crucial lower bound, see \eqref{crucial lower bd} below. The main technical innovation is that, by leveraging the additional regularity of $V$, we can decrease $\varphi'$ to zero (outside of a compact set) in an explicit,  $E$-dependent fashion. We then obtain \eqref{ext est high dim} for any $M$ such that $\mathbf{1}_{|x| \ge M}$ is supported in the set where $\varphi$ is constant.

If we do not assume anything about the derivatives of $V$, for instance, if $V \in L_{\text{comp}}^\infty(\R^n ; \R)$, then the best known bound in general is $g_s^\pm \le \exp( Ch^{-4/3}  \log(h^{-1}))$ \cite{klvo19,sh20}, although Vodev \cite{vo20c} showed this can be improved to $g_s^\pm \le \exp(Ch^{-4/3})$ if $V$ is short range and radial. See also \cite{vo19a, vo20a, vo20b, gash20}. On the other hand, it is not known whether an exterior estimate like \eqref{ext est high dim} holds for  $L^\infty$ potentials, except in dimension one \cite{dash20}, and there $\mathbf{1}_{|x| \ge M}$ and $V$ need only have disjoint supports. 

We remark that the Theorem should hold in dimension two also, provided  $V \in L^\infty$ and $|\nabla V|$ is locally bounded near the origin. The extra difficulty in dimension two comes from the effective potential term, see \eqref{Lambda positive} below, having a negative singularity at $r = 0$. This necessitates a stronger assumption on the derivatives of $V$, see \cite[Theorem 4.2]{vo20b} for more details. 

For more background on semiclassical resolvent estimates, we refer the reader to the introductions of \cite{daji20, gash20}.

\bigskip

\noindent {\sc Acknowledgements.}
The authors thank  Kiril Datchev for helpful comments and for reading an early version of this article, and the two anonymous referees for their careful reading and insightful remarks. J. Shapiro was supported in part by the Australian Research Council through grant DP180100589.\\

\section{Notation and preliminary calculations}

\textbf{Notation:} Throughout, ``prime" notation indicates differentiation with respect to the radial variable $r  =|x|$, e.g., $u' = \partial_r u$. 

As in most previous proofs of resolvent estimates for low regularity potentials, the backbone of the proof is a Carleman estimate. Our Carleman estimate is stated as Lemma \ref{Carleman lemma}.

 We start from the identities 
\begin{equation}
\begin{gathered}
    r^{\frac{n-1}{2}}(- \Delta) r^{-\frac{n-1}{2}} = -\partial^2_r + \Lambda,
    \label{Lambda positive}\\
    \Lambda \defeq \frac{1}{r^2}\left( -\Delta_{\US^{n-1}} + \frac{(n-1)(n-3)}{4} \right) \ge 0,
    \end{gathered}
\end{equation}
where $\Delta_{\US^{n-1}}$ denotes the negative Laplace-Beltrami operator on $\US^{n-1}$. Below, we construct an absolutely continuous phase function $\varphi$ on $[0, \infty)$ which obeys $\varphi \ge 0$, $\varphi(0) = 0,$ and $\varphi' \ge 0$. Using $\varphi$, we form the conjugated operator
\begin{equation} \label{conjugation}
\begin{split}
  P^{\pm}_\varphi(h) &\defeq e^{\varphi/h} r^{\frac{n-1}{2}}\left( P(h) - E \pm i\varepsilon \right) r^{-\frac{n-1}{2}} e^{-\varphi/h}\\
  &= -h^2\partial^2_r + 2h \varphi' \partial_r + h^2\Lambda + V -(\varphi')^2 + h\varphi''  - E \pm i\varepsilon.
 \end{split}
\end{equation}

Let $u \in e^{\varphi/h} r^{(n-1)/2} C^\infty_0(\R^n)$ when $n \ge 3$, $u \in e^{\varphi/h} r^{(n-1)/2} \mathcal{D}(P)$ when $n = 1$. Define a spherical energy functional $F[u](r)$,
\begin{equation} \label{F}
    F(r) = F[u](r) \defeq \|hu'(r, \cdot)\|^2 - \langle (h^2\Lambda + V  -(\varphi')^2- E)u(r, \cdot), u(r, \cdot) \rangle,
\end{equation}
where $\| \cdot \|$ and $\langle \cdot, \cdot \rangle$ denote the norm and inner product on $L^2(\mathbb{S}_\theta^{n-1})$, respectively (when $n = 1$, $\|u(r, \theta)\|_{\US^{0}} \defeq |u(r)| + |u(-r)|)$. It is easy to compute (see e.g.~\cite{da14,sh19, sh20, gash20}) that for $w\in C^0[0,\infty)$ and piecewise $C^1$, 
$(wF)'$, as a distribution on $(0,\infty)$, is given by 
\begin{equation} \label{deriv wF}
\begin{split}
    (wF)'
    &= -2 w\real \langle P^{\pm}_\varphi(h) u, u' \rangle \mp 2\varepsilon w \imag \langle u,u'\rangle + (2wr^{-1} - w') \langle h^2\Lambda u,u\rangle \\
    &+ (4h^{-1}w \varphi' + w')\|hu'\|^2  + (w(E+ (\varphi')^2- V ))' \|u\|^2 + 2w\real \langle  h \varphi'' u, u' \rangle. 
    \end{split}
\end{equation}
We will construct $w \ge 0$, $w(0) = 0$ such that
\begin{equation}
\label{e:weightCond}
2wr^{-1}-w'\geq 0,
\end{equation} 
and use ~\eqref{Lambda positive}
to control the term involving $\Lambda$. Using~\eqref{e:weightCond} together with $2ab \ge -(\gamma a^2 + \gamma^{-1}b^2)$ for $\gamma > 0$, we find
\begin{equation*}
\begin{split}
    (wF)'
    &= -2 w\real \langle P^{\pm}_\varphi(h) u, u' \rangle \mp 2\varepsilon w \imag \langle u,u'\rangle + (2wr^{-1} - w') \langle h^2\Lambda u,u\rangle \\
    &+ (4h^{-1}w \varphi' + w')\|hu'\|^2  + (w(E+ (\varphi')^2- V ))' \|u\|^2 + 2w\real \langle  h \varphi'' u, u' \rangle \\
    &\ge -\frac{\gamma_1 w^2}{h^2w'} \| P^{\pm}_\varphi(h) u \|^2 \mp 2 \varepsilon w \imag \langle u, u' \rangle \\
     &+ (4(1 - \gamma^{-1}_2) h^{-1}w \varphi' + (1 - \gamma^{-1}_1 - \gamma^{-1}_2) w')  \|hu'\|^2 \\
    & + (w(E+ (\varphi')^2- V ))' \|u\|^2 - \frac{\gamma_2(w \varphi'')^2}{w' + 4h^{-1}\varphi' w} \|u\|^2, \qquad \gamma_1, \, \gamma_2 > 0. 
    \end{split}
\end{equation*}
Fix $\eta > 0$ and put $\gamma_1 = (1 + \eta)/\eta, \, \gamma_2 = 1 + \eta$, yielding
\begin{equation}
\label{e:lowerDerivative}
\begin{split}
    (wF)'
      &\ge -\frac{(1 + \eta)w^2}{\eta h^2w'} \| P^{\pm}_\varphi(h) u \|^2 \mp 2 \varepsilon w \imag \langle u, u' \rangle \\
    & + (w(E+ (\varphi')^2- V ))' \|u\|^2 - \frac{(1 + \eta)(w \varphi'')^2}{w' + 4h^{-1}\varphi' w} \|u\|^2. 
    \end{split}
\end{equation}
To complete the proof of the Carleman estimate, we seek to build $w$ and $\varphi$ so that the second line of \eqref{e:lowerDerivative} has a good lower bound. Indeed, putting 
\begin{equation} \label{A and B}
A(r) \defeq (w(E+ (\varphi')^2-V))',\qquad B(r)\defeq \frac{(w\varphi'')^2}{w' + 4h^{-1}\varphi' w},
\end{equation}
it suffices for $w$ and $\varphi$ to satisfy, for $\eta > 0$ fixed, 
\begin{equation} \label{goal est}
A(r)-(1+\eta)B(r) \geq \frac{E}{2}w', \qquad 0 < h \ll 1,
\end{equation}
along with a few other properties (see \eqref{W} through \eqref{phi prime}).

In order to construct the weight and phase functions for our Carleman estimates, we adapt the method in~\cite{gash20}. Whenever $w',\varphi' \neq 0$, put 
\begin{equation}\label{e:defPhiW}
\Phi:=\frac{\varphi''}{\varphi'}=(\log |\varphi'|)',\qquad \mathcal{W}:=\frac{w}{w'}=\frac{1}{(\log |w|)'},
\end{equation}
Then, as in~\cite[(2.10)]{gash20},
\begin{equation}
\label{key prelim calc}
\begin{aligned}
A(r)- (1+\eta)B(r)&\geq w'\Big[ E+(\varphi')^2(1+2\mathcal{W}\Phi-(1+\eta)\mathcal{W}\Phi^2\min(\mathcal{W},\tfrac{h}{4\varphi'})) -V-\mathcal{W}V'  \Big].
\end{aligned}
\end{equation}
So when $|w'|, \, |\varphi'| > 0$, to show \eqref{goal est}, it is enough to bound the bracketed expression in \eqref{key prelim calc} from below by $E/2$.

\section{Construction of the phase and weight functions} \label{phase and weight section}

Throughout this section, we assume $E > 0$, $s > 1/2$ are fixed, and suppose $V$ satisfies \eqref{V bound r small} through \eqref{V prime bound r large}. Using \eqref{V bound r large} and \eqref{V prime bound r large}, let
\begin{equation} \label{defn b}
b \defeq \sup\{ |x| \ge 1 : V(x)+ |x| V'(x) \geq \tfrac{E}{4}\} < \infty,
\end{equation}
so that $b$ is independent of $h$ and
\begin{equation} \label{V split b}
(V + |x| V')\mathbf{1}_{|x| \ge 1} \le (V +|x| V')\mathbf{1}_{|x| \le b}  + \frac{E}{4} \mathbf{1}_{|x| > b}.
\end{equation}
(Note that $b$ can be chosen to depend only on $p$, $m$, $c_0$, and $E$, and that $b\leq R_0$  provided $\supp V \subseteq B(0,R_0)$.)
Additionally, let 
\begin{equation} \label{a M tau}
M > a \ge b, \qquad  \tau_0 \ge 1,
\end{equation}
be parameters, independent of $h$, to be specified in the proof of Lemma \ref{crucial lower bd lemma} below.

Let $\omega\in C_{0}^\infty((-3/4,3/4);[0,1])$ with $\omega = 1 $ near $[-1/2,1/2]$. The weight $w$ and phase $\varphi$, which will  be shown to satisfy \eqref{goal est}, are functions of the radial variable $r = |x|$ only, and are defined by
\begin{gather} 
\tilde{m}(r) \defeq \min \Big[ \frac{E}{2c_0} m^{-1}(r), (r + 1)^{2s-1} \Big], \label{tilde m} \\
w(0) = 0, \, w'(0) = 1, \qquad \frac{w}{w'} = \mathcal{W} \defeq \begin{cases} \frac{r(1+\omega(r))}{2} & 0 < r < M \\
\frac{r}{2} \tilde{m}(r) & r > M
\end{cases} \label{W}, \\
 \varphi(0) = 0, \, \varphi'(\tfrac{1}{2}) = \tfrac{2}{3} \tau_0, \qquad \frac{\varphi''}{\varphi'} = \Phi \defeq  \begin{cases} -\frac{\delta}{2r}&0<r<\frac{1}{2}\\ -\frac{1}{r + 1} & \tfrac{1}{2} < r < a\\
-\frac{2}{M -r} & a < r < M
 \end{cases},\qquad
 \varphi' = 0, \qquad r \ge M. \label{Phi}
\end{gather}
 Short computations yield,
\begin{gather} 
w = \begin{cases} r & 0 < r < 1/2 \\
 \frac{1}{2}e^{\int_{1/2}^r \frac{2}{s(1+\omega(s))}ds} & 1/2 \le r < M \\
w(M) e^{\int_M^r \frac{2}{s \tilde{m}(s)} ds} & r \ge M \end{cases} ,\qquad
  w' = \begin{cases} 1 & 0 < r < 1/2 \\
  \frac{1}{r(1+\omega(r))}w&1/2 <r<M \\
   \frac{2w(M) e^{\int_M^r \frac{2}{s\tilde{m}(s)} ds}}{r \tilde{m}(r)} \ge (r + 1)^{-2s} & r > M
   \end{cases}, \label{w}  \\
\varphi'  = \begin{cases}  3^{-1}\cdot 2^{-\frac{\delta}{2} +1}\tau_0r^{-\frac{\delta}{2}}&0<r<\tfrac{1}{2}\\
 \frac{\tau_0}{r+1} & \tfrac{1}{2} \le r < a \\
\varphi'(a) \Big( \frac{M - r}{M - a} \Big)^2 & a \le r < M \\
0 & r \ge M
  \end{cases}.  \label{phi prime}
  \end{gather}


We now prove the crucial lower bound involving $E$, $w$ and $\varphi$ that is needed to prove the Carleman estimate.

\begin{lemma} \label{crucial lower bd lemma} 
Fix $0<\eta<\frac{4-4\delta-\delta^2}{\delta^2}$ and let $V$ satisfy \eqref{V bound r small}, \eqref{V bound r large}, \eqref{V prime bound r small} and \eqref{V prime bound r large}. Then, using the notation of \eqref{A and B} and \eqref{tilde m} through \eqref{phi prime},  there exist suitable $M, \, a$ and $\tau_0$ so that   
\begin{equation} \label{crucial lower bd}
A - (1+\eta)B \ge \frac{E}{2}w' , \qquad h \in (0, h_0], \, r > 0, \, r \neq \frac{1}{2}, \, a, \, M.
\end{equation}
\end{lemma}

Once Lemma~\ref{crucial lower bd lemma} is proved, we can use a standard argument similar to that found, e.g., in~\cite[Sections 5,6]{gash20} to prove the following Carleman estimate. This argument is contained in Section~\ref{Carleman estimates}.
\begin{lemma} \label{Carleman lemma}
There are $C, \, h_0 > 0$ independent of $h$ and $\varepsilon$ so that 
\begin{equation} \label{Carleman est}
    \| \langle x \rangle^{- s} e^{\varphi/h} v \|^2_{L^2} \le \frac{C}{h^2} \|\langle x \rangle^{s}e^{\varphi/h}(P(h) - E \pm i \varepsilon)v \|^2_{L^2} + \frac{C \varepsilon}{h} \| e^{\varphi/h} v \|^2_{L^2},
\end{equation}
for all $\varepsilon > 0$ and $h \in (0, h_0]$, and for all $v \in C^\infty_0(\mathbb{R}^n)$ ($n \ge 3$) or all $v \in \mathcal{D}(P)$ with $\langle x \rangle^sPv \in L^2(\R)$ ($n =1$).
\end{lemma}
From here, Theorem \ref{thm high dim} follows from the argument in Section~\ref{s: resolv ests}.

\begin{proof}[Proof of Lemma \ref{crucial lower bd lemma}]

\textbf{Case $0 < r < \tfrac{1}{2}$:}\\
Using~\eqref{key prelim calc}, together with $\mathcal{W}=r$, $\Phi=-\frac{\delta}{2r}$, and $\varphi'= 3^{-1}\cdot 2^{-\frac{\delta}{2} +1}\tau_0r^{-\frac{\delta}{2}}$, we have
\begin{equation*}
\begin{split}
   A(r)-(1+\eta)B(r) &\ge w'\Big[ E+(\varphi')^2(1+2\mathcal{W}\Phi- (1+\eta)\mathcal{W}\Phi^2\min(\mathcal{W},\tfrac{h}{4\varphi'})) -V-\mathcal{W}V'  \Big]\\
   &\ge w' \left[ E + 9^{-1} \cdot 2^{-\delta +2} \tau^2_0 r^{-\delta} \left( 1 - \delta  - \tfrac{1}{4}(1 + \eta) \delta^2\right) -2c_1 r^{-\delta}  \right]. 
    \end{split}
\end{equation*}
Since $\eta>0$ is such that   $1- \delta - 4^{-1}(1 + \eta)\delta^2 >0$, choosing
\begin{equation}
\tau^2_0 \geq \frac{9 \cdot 2^{\delta -1}c_1}{(1 - \delta - 4^{-1}(1 + \eta)\delta^2)},
\end{equation}
yields
\begin{equation*}
A - (1 + \eta)B \ge \frac{E}{2} w', \qquad h \in (0,1], \, 0 < r < \frac{1}{2}.
\end{equation*}

The precise value, $(1+\eta)$ will not play a crucial role below, therefore, to ease notation, we put $K:=1+\eta$ below.

\textbf{Case $\tfrac{1}{2} < r < a$:}\\
First, recall~\eqref{key prelim calc}:
\begin{equation*}
\begin{aligned}
A(r)- KB(r)&\geq  w'\Big[ E+(\varphi')^2(1+2\mathcal{W}\Phi-K\mathcal{W}\Phi^2\min(\mathcal{W},\tfrac{h}{4\varphi'})) -V-\mathcal{W}V'  \Big].
\end{aligned}
\end{equation*}
By \eqref{W} and \eqref{Phi}, 
\begin{equation*}
 1 + 2 \mathcal{W} \Phi \ge \frac{1}{4(r+1)}, \qquad \tfrac{1}{2} < r < a.
\end{equation*}
Also by~\eqref{W}, $|\mathcal{W}|\leq r$ when $0 < r < a$, hence appealing to \eqref{V split b},
$$
V+\mathcal{W}V'\leq ( V +  r V') \mathbf{1}_{|x| \le b} + \frac{E}{4} \mathbf{1}_{|x|> b}.
$$
Furthermore, using $|\mathcal{W}|\leq r$ again, by~\eqref{Phi} $\Phi^2=(r+1)^{-2}$, and by~\eqref{phi prime} $\varphi'=\tau_0(r+1)^{-1}$,
\begin{equation*}
(\varphi')^2\mathcal{W}\Phi^2 \min(\mathcal{W},\tfrac{h}{4\varphi'})\leq \frac{hr\tau_0}{4(r+1)^3}, \qquad 0 < r < a.
\end{equation*}

From these estimates, and using once more that $\varphi'=\tau_0(r+1)^{-1}$, we find,
\begin{equation} \label{A - KB prelim Lip r small}
\begin{aligned}
A(r)&- KB(r)\\
&\ge w' \Big[ E+(\varphi')^2 \big( 1+2\mathcal{W}\Phi \big)
- \frac{ K h r\tau_0}{4(r+1 )^{3}} 
- ( V + r V') \mathbf{1}_{r \le b} - \frac{E}{4} \mathbf{1}_{r > b} \Big] \\
&\ge w' \Big[ \frac{3E}{4} + \Big( \frac{\tau^2_0}{4(r+1)^{3}} - 2c_1 r^{-\delta} \mathbf{1}_{0 < r  < 1} - (V + rV')\mathbf{1}_{1 \le r \le b}  \Big)   -     \frac{Khr\tau_0}{4(r+1)^3} \Big], \qquad \tfrac{1}{2} < r < a.
\end{aligned}
\end{equation}

We now further increase $\tau_0$, if necessary, so that 
$$\tau_0 \ge 2\sup_{\frac{1}{2}\leq r\leq b}\, (r+1)^{3/2}\sqrt{ 2c_1 r^{-\delta}  \mathbf{1}_{0 < r  < 1} + (p(r)+c_0m(r))\mathbf{1}_{1 \le r \le b}},
$$
 which makes the term in parenthesis in the third line of \eqref{A - KB prelim Lip r small} is nonnegative.
We then take $h_0 = h_0(K, \tau_0, E) \in (0, 1]$ sufficiently small to achieve
\begin{equation} \label{key est small r Lip}
A- K B \ge  \frac{E}{2}w', \qquad  h \in (0,h_0], \, \frac{1}{2} < r < a.
\end{equation}
\noindent \textbf{Case $a < r  < M$:}\\
As in the previous case, we begin from \eqref{key prelim calc}. We use~\eqref{W}, \eqref{Phi} and \eqref{phi prime} to see
$$
(\varphi')^2(1+2\mathcal{W}\Phi)=(\varphi'(a))^2\Big(\frac{M-r}{M-a}\Big)^4\Big(1-\frac{2r}{M-r}\Big).
$$
Next, we use $a\geq b$, $|\mathcal{W}|\leq r/2$ and~\eqref{V split b} to obtain 
$$
V+\mathcal{W}V'\leq \frac{E}{4} \mathbf{1}_{> b}.
$$
Then, again by~\eqref{W},~\eqref{Phi}, and~\eqref{phi prime},
\begin{equation*}
(\varphi')^2\mathcal{W}\Phi^2\min(\mathcal{W},\tfrac{h}{4\varphi'})\leq \frac{hr\varphi'(a)}{2(M-a)^2}.
\end{equation*}
Combining these bounds with~\eqref{key prelim calc} and the formula~\eqref{phi prime} for $\varphi'(a)$, we have
\begin{equation} \label{A - KB prelim Lip r big}
\begin{split}
A(r)- KB(r)& \ge w' \Big[  \frac{3 E}{4} +(\varphi'(a))^2 \Big( \frac{M - r}{M - a} \Big)^{4} \Big(  1 -  \frac{2r}{M -r } \Big)  - 2^{-1} K    h \varphi'(a)   r \frac{1}{(M -a)^{2}} \Big] \\
&\ge w' \Big[  \frac{3 E}{4} - \frac{2\tau_0^2}{(a+1)^2} r  \frac{(M-r)^{3}}{(M- a)^{4}}   - 2^{-1} K \frac{\tau_0 }{a+1}  h r \frac{1}{(M -a)^{2}}\Big].
\end{split}
\end{equation}
Now, choose $M = 2a$ and estimate, for $a < r < M$,
\begin{equation*} 
2r \frac{(M-r)^{3}}{(M- a)^{4}} \le 4,\qquad  r \frac{1}{(M -a)^2} \le \frac{2}{a}. 
\end{equation*}

Therefore, we choose 
$$
a=\max(\sqrt{20}\tau_0E^{-1/2},b)
$$
and $h_0= K^{-1}$, ensuring that the bracketed terms in the second line of \eqref{A - KB prelim Lip r big}  are bounded from below by $E/2$. This yields
\begin{equation} \label{key est large r Lip}
A- K B \ge  \frac{E}{2}w', \qquad  h \in (0,h_0], \, a < r < M.
\end{equation}
 \noindent \textbf{Case $r  > M$:}\\
 In this final case we have $\varphi' = 0$, so appealing to \eqref{A and B}, we have
 $$
 A-KB=w' \big[ E- V - \mathcal{W} V' ].
 $$
 By \eqref{V split b}, $V\leq \frac{E}{4}$. By \eqref{tilde m} and~\eqref{W}, $\mathcal{W}V'\leq c_0 m \tilde{m}/2 \le E/4$. Hence, 
 \begin{equation*}
 A- KB = w' \big[ E- V - \mathcal{W} V' ] \ge w' \Big[\frac{3E}{4} - \frac{c_0 m \tilde{m}}{2} \Big] \ge \frac{E}{2} w',  \qquad  h \in (0,1], \, r > M.
 \end{equation*}
 This completes the proof of the Lemma. \\
\end{proof}

\section{Carleman estimate} \label{Carleman estimates}
Our goal in this section is to prove Lemma \ref{Carleman lemma}. This argument is standard; versions of it appear, for instance, in \cite[Section 5]{gash20} in \cite[proof of Lemma 2.2]{da14}, but we include it here for the reader's convenience.  

\noindent {\bf{Remark:}} In the proof of Lemma \ref{Carleman lemma}, we abuse notation slightly. In dimension $n\geq 3$, we put $\|u\|=\|u(r, \cdot)\|_{L^2(\US_\theta^{n-1})}$, while we put $\|u\|=|u(r)| + |u(-r)|$ when $n=1$. If $n \ge 3$, $\int_{r,\theta}$ denotes the integral over $(0,\infty) \times \US^{n-1}$ with respect to the measure $dr d\theta$, while if $n =1$, $\int_{r,\theta}u(x)$ denotes $\int_0^\infty u(r) dr + \int^\infty_0 u(-r) dr = \int_\R u(x) dx$. 
\smallskip
\begin{proof}[Proof of Lemma \ref{Carleman lemma}]
Since $\langle x \rangle^{-2s} \le 1$, without loss of generality, we may assume $0 < \varepsilon \le h$.

The proof begins from~\eqref{e:lowerDerivative}. Then, applying~\eqref{crucial lower bd}, it follows that, for $h \in (0,h_0]$,
\begin{equation} \label{lower bound wF prime}
   w' F + w F' \ge -\frac{3 w^2}{h^2w'} \| P^{\pm}_\varphi(h)u \|^2 \mp  2\varepsilon w \imag \langle u,u'\rangle + \frac{1}{3}w'\|hu' \|^2 +\frac{E}{2} w' \|u\|^2.
\end{equation}

Now we integrate both sides of \eqref{lower bound wF prime}. We integrate $\int^\infty_0dr$ and use\\ $wF, \, (wF)' \in L^1((0,\infty);dr)$ (when $n =1$, this follows from the facts that $w$ is bounded and $w/w' \le (r +1)^{2s}$). Since, $wF(0) =0$, $\int_{0}^\infty (wF)'dr = 0$. Therefore, 
\begin{equation}
    \int_{r,\theta} w' \left(|u|^2 +|hu'|^2 \right) \lesssim \frac{1}{h^2} \int_{r,\theta} (r + 1)^{2s}|P^{\pm}_\varphi(h)u|^2  +  2\varepsilon \int_{r, \theta} w|u u'|, \qquad h \in (0,h_0].
    \end{equation}
In addition to $w$ bounded, $w/w' \le (r +1)^{2s}$, \eqref{w}, also gives $w'\geq w$ on $r<\frac{1}{2}$ and $w' \ge (r + 1)^{-2s}$. Therefore,
\begin{equation} \label{penult est}
    \int_{r,\theta} (r + 1)^{-2s}\left(|u|^2 +|hu'|^2 \right) \lesssim \frac{1}{h^2} \int_{r,\theta} (r + 1)^{2s}|P^{\pm}_\varphi(h)u|^2  +  \frac{\varepsilon}{h} \int_{r, \theta} |u|^2 + \frac{\varepsilon}{h} \int_{\substack{r,\theta\\r\geq \frac{1}{2}}}|hu'|^2, \qquad h \in (0,h_0],
\end{equation}
where we have used $0 < \varepsilon \le h$.

Moreover, letting $\chi \in C_0^\infty[0,\frac{1}{2})$ with $\chi \equiv 1$ on $[0,\frac{1}{4}]$ and $\psi(r)=1-\chi(r)$,
\begin{equation} \label{rewrite Pu bar u}
    \begin{split}
        \real \int_{r,\theta}(P^{\pm}_\varphi(h) \psi u)\overline{\psi u} &= \int_{r,\theta} |h(\psi u)'|^2  + \real \int_{r,\theta} 2 h\varphi' (\psi u)' \overline{\psi u} + \int_{r, \theta} (h^2 \Lambda \psi u)\overline{\psi u} \\
        &+ \int_{r,\theta} h\varphi''|\psi u|^2 + \int_{r,\theta} \left(V - E - (\varphi')^2 \right)|\psi u|^2,\\
    \end{split}
\end{equation}
and 
\begin{equation} \label{int by parts}
    \int_{r, \theta} h \varphi''|\psi u|^2 = - \real \int_{r,\theta} 2 \varphi'h (\psi u)' \overline{\psi u}. 
\end{equation}
These two identities, together with the facts that  $\Lambda \ge 0$ and $|V - E - (\varphi')^2|$ is bounded on $\supp \psi$ (independent of $h$ and $\varepsilon$) imply, that for all $h\in (0,1]$, $\gamma>0$,
\begin{equation} \label{handle deriv u term}
\begin{split}
    \int_{\substack{r, \theta\\ r\geq \frac{1}{2}}} |hu'|^2& \lesssim  \int_{r,\theta} |u|^2 + \frac{\gamma}{2} \int_{r, \theta}(r +1)^{-2s} |u|^2 + \frac{1}{2\gamma} \int_{r, \theta}(r +1)^{2s} |P^{\pm}_\varphi(h)(\psi u)|^2\\
    &\leq \int_{r,\theta} |u|^2 + \frac{\gamma}{2} \int_{r, \theta}(r +1)^{-2s} |u|^2 + \frac{1}{\gamma} \int_{r, \theta}(r +1)^{2s} |\psi P^{\pm}_\varphi(h)( u)|^2+ \frac{
    Ch^2}{\gamma} \int_{\substack {r, \theta\\\frac{1}{4}\leq r\leq \frac{1}{2}}} |hu'|^2. 
    \end{split}
\end{equation}
To finish, we substitute \eqref{handle deriv u term} into the right side of \eqref{penult est}. Recalling $0 < \varepsilon \le h$, choose $\gamma > 0$ small enough (independent of $h$ and $\varepsilon$), and then $h$ small enough, to absorb the second and fourth terms in the last line of~\eqref{handle deriv u term} into the right side of \eqref{penult est}. We obtain
\begin{equation} \label{final est}
\begin{split}
    \int_{r,\theta} &(r + 1)^{-2s}(|u|^2 + |hu'|^2) \lesssim \\
    & \frac{1}{h^2} \int_{r,\theta} (r + 1)^{2s}|P^{\pm}_\varphi(h)u|^2  +  \frac{\varepsilon}{h} \int_{r, \theta} |u|^2, \qquad h \in (0, h_0].
    \end{split}
\end{equation}
The estimate
  \eqref{Carleman est} is now an easy consequence of \eqref{final est}.\\
\end{proof}

\section{Resolvent estimates} \label{s: resolv ests}
In this section, we deduce Theorem \ref{thm high dim} from Lemma \ref{Carleman lemma}. The same argument appears, e.g., in \cite[Section 6]{gash20} and \cite[last proof of Section 2]{da14}, but we include it here for completeness.

\begin{proof}[Proof of the Theorem]
Since increasing $s$ in only decreases the weighted resolvent norms \eqref{defn g} and \eqref{defn g a}, without loss of generality we may take $1/2 < s < 1$. Let $C, h_0 > 0$ be as in the statement of Lemma \ref{Carleman lemma}. Put $C_\varphi = C_\varphi(h) \defeq 2 \max \varphi$. By \eqref{Carleman est}, and since $2\varphi(r) = C_\varphi$ for $r \ge M$,
\begin{equation} \label{mult through by exp}
\begin{split}
e^{-C_\varphi/h}\|\langle x \rangle^{-s} \mathbf{1}_{\le M} v \|^2_{L^2} + \|\langle x \rangle^{-s} \mathbf{1}_{\ge M} v \|^2_{L^2}  &\le 
 \frac{C}{h^2} \|\langle x \rangle^{s} (P(h) - E \pm i\varepsilon)v \|^2_{L^2}  +  \frac{C\varepsilon}{h}  \| v \|^2_{L^2},
\end{split}
\end{equation}
for all $\varepsilon \ge 0$ and $h \in (0, h_0]$, and all $v \in C^\infty_0(\mathbb{R}^n)$ ($n \ge 3$) or all $v \in \mathcal{D}(P)$ with $\langle x \rangle^sPv \in L^2(\R)$ ($n =1$). Moreover, for any $\gamma, \, \gamma_0 > 0$,
\begin{equation} \label{epsilon v}
\begin{split}
2\varepsilon \| v \|^2_{L^2} &= -2 \imag\langle (P(h) - E \pm i\varepsilon)v, v \rangle_{L^2} 
\\& \le  \gamma^{-1}\|\langle x \rangle^{s} \mathbf{1}_{\le M} (P(h) - E \pm i \varepsilon)v \|^2_{L^2} 
+ \gamma\|\langle x \rangle^{-s}  \mathbf{1}_{\le M}  v\|^2_{L^2} \\
   &+\gamma_0^{-1}\|\langle x \rangle^{s} \mathbf{1}_{\ge M} (P(h) - E \pm i \varepsilon)v \|^2_{L^2} 
+ \gamma_0\|\langle x \rangle^{-s}  \mathbf{1}_{\ge M}  v\|^2_{L^2}
\end{split}
\end{equation}
 Setting $\gamma = he^{-C_\varphi/ h}/C $ and $\gamma_0 = h/C$, \eqref{mult through by exp} and \eqref{epsilon v} imply, for some $\tilde{C} > 0$, all $\varepsilon \ge 0, \, h \in (0,h_0]$, and all $v \in C^\infty_0(\mathbb{R}^n)$ ($n \ge 3$) or all $v \in \mathcal{D}(P)$ with $\langle x \rangle^sPv \in L^2(\R)$ ($n =1$),
\begin{equation} \label{penult}
\begin{split}
e^{-\tilde{C}/h} \|\langle x \rangle^{-s}  \mathbf{1}_{\le M} v \|_{L^2}^2 + \|\langle x \rangle^{-s} \mathbf{1}_{\ge M} v \|_{L^2}^2 
& \le e^{\tilde{C}/h} \|\langle x \rangle^s \mathbf{1}_{\le M} (P(h) - E \pm i \varepsilon)v \|_{L^2}^2 \\
 &+ \frac{\tilde{C}}{h^2} \|\langle x \rangle^s \mathbf{1}_{\ge M} (P(h) - E \pm i \varepsilon)v \|_{L^2}^2.
\end{split}
\end{equation}

The final task is to use \eqref{penult} to deduce 
\begin{equation} \label{ult}
\begin{gathered}
e^{-\tilde{C}/{h}} \|\langle x  \rangle^{-s}  \mathbf{1}_{\le M} (P(h)-E \pm i\varepsilon)^{-1} \langle x \rangle^{-s} f \|^2_{L^2} +  \|\langle x  \rangle^{-s}  \mathbf{1}_{\ge M} (P(h)-E \pm i\varepsilon)^{-1} \langle x \rangle^{-s} f \|^2_{L^2} \\
\leq e^{\tilde{C}/h} \| \mathbf{1}_{\le M} f \|_{L^2}^2 + \frac{\tilde{C}}{h^2}  \| \mathbf{1}_{\ge M} f \|_{L^2}^2 , \qquad  \varepsilon > 0, \, h \in (0,h_0], \, f \in L^2,
\end{gathered}
\end{equation}
from which Theorem \ref{thm high dim} follows. If $n =1$, \eqref{penult} immediately implies \eqref{ult} by setting $v = (P(h) - E \pm i \varepsilon)^{-1} \langle  x \rangle^{-s} f$. To establish \eqref{ult} when $n \ge 3$, we prove a simple Sobolev space estimate and then apply a density argument that relies on \eqref{penult}. 

The operator
\begin{equation*}
[P(h), \langle x \rangle^s]\langle x \rangle^{-s} = \left(-h^2 \Delta \langle x \rangle^s - 2h^2 (\nabla \langle x \rangle^s) \cdot \nabla \right) \langle x \rangle^{-s}
\end{equation*}
is bounded $H^2 \to L^2$. So, for $v \in H^2$ such that $\langle x \rangle^s v \in H^2$,
 \begin{equation}\label{Ceph}
 \begin{split}
\|\langle x \rangle^{s}(P(h)-E \pm i \varepsilon)v\|_{L^2}  &\le \|(P(h)-E \pm i \varepsilon)\langle x \rangle^{s} v \|_{L^2} +  \|[P(h),\langle x \rangle^{s}]\langle x \rangle^{-s}\langle x \rangle^{s}v \|_{L^2}
\\& \le C_{\varepsilon, h} \| \langle x \rangle^{s}v \|_{H^2},
\end{split} 
\end{equation}
for some constant $C_{\varepsilon, h}>0$ depending on $\varepsilon$ and $h$.

Given $f \in L^2$, the function $ u= \langle x \rangle^{s}(P(h)-E\pm i\varepsilon)^{-1}\langle x \rangle^{-s} f \in H^2$ because 
\begin{equation*}
u = (P(h) - E \pm i \varepsilon)^{-1} (f - w), \qquad w = \langle x \rangle^{s} [P(h), \langle x \rangle^{-s}]   \langle x \rangle^{s}  \langle x \rangle^{-s}u,
\end{equation*}
with  $\langle x \rangle^{s} [P(h), \langle x \rangle^{-s}]   \langle x \rangle^{s}$ being bounded $L^2 \to L^2$ since $s < 1$.

Now, choose a sequence $v_k \in C_{0}^\infty$ such that $ v_k \to  \langle x \rangle^{s}(P(h)-E \pm i\varepsilon)^{-1}\langle x \rangle^{-s} f$ in $H^2$. Define \\ $\tilde{v}_k \defeq \langle x \rangle^{-s}v_k$. Then, as $k \to \infty$,
\begin{equation*}
\begin{split}
\| \langle x \rangle^{-s} \tilde{v}_k - \langle x \rangle^{-s} (&P(h)-E\pm i \varepsilon)^{-1}\langle x \rangle^{-s}f \|_{L^2}  \\
&\le \| v_k - \langle x \rangle^{s} (P(h)-E\pm i \varepsilon)^{-1}\langle x \rangle^{-s}f \|_{H^2} \to 0.
\end{split}
\end{equation*}
Also, applying \eqref{Ceph},
\begin{equation*}
\|\langle x \rangle^{s}(P(h)-E\pm i \varepsilon)\tilde v_k - f\|_{L^2} \le C_{\varepsilon,h} \|v_k - \langle x \rangle^{s} (P(h)-E \pm i \varepsilon)^{-1} \langle x \rangle^{-s} f \|_{H^2} \to 0.
\end{equation*} 
We then achieve \eqref{ult} by replacing $v$ by $\tilde{v}_k$ in \eqref{penult} and sending $k \to \infty$.\\
\end{proof}

\appendix
\section{Self-adjointness of $P$ in one dimension}

\begin{lemma}[generalization of the Lemma in \cite{dash20}]
Suppose $V : \R \to \R$ may be written as $V = V_1 + V_\infty$ for $V_1 \in L^1(\R)$ and $V_\infty \in L^\infty(\R)$. Let $\mathcal D$ be the set of all $u \in H^1(\R)$ such that $u' \in L^\infty(\R)$ and $Pu \in L^2(\R)$. Then $P$, with domain $\mathcal{D}$, is densely defined and self-adjoint on $L^2(\R)$.
\end{lemma}

\begin{proof}[Proof of Lemma]

Let $\mathcal D_\textrm{max}$ be the set of all $u \in L^2$ such that $u' \in L^1_\textrm{loc}$ and $Pu \in L^2$. By \cite[Lemma 10.3.1]{zet}, $\mathcal{D}_{\text{max}}$ is dense in $L^2(\R)$. We  begin by proving that  $\mathcal D_\textrm{max} = \mathcal D$. Indeed, for any $a>0$ and $u \in \mathcal D_\textrm{max}$, by integration by parts and Cauchy--Schwarz, we have
\begin{equation*}
\begin{gathered}
\int_{-a}^a|u'|^2 = u'\bar u|_{-a}^a - \int_{-a}^a u''\bar u \le \\
 2\sup_{[-a,a]}|u'|\sup_{[-a,a]} |u| + h^{-2}\|V_1\|_{L^1} \sup_{[-a,a]}|u|^2 + h^{-2}(\| V_\infty\|_{L^\infty} + \|Pu\|_{L^2})\|u\|_{L^2},\\
\sup_{[-a,a]}|u|^2 =\sup_{x \in [-a,a]}  \left(|u(0)|^2 + 2 \real \int_0^x u'\bar u\right)  \le |u(0)|^2 + 2\left(\int_{-a}^a|u'|^2\right)^{1/2}\|u\|_{L^2},\\
\sup_{[-a,a]} |u'|^2 \le |u'(0)|^2 + 2 h^{-2}\|V_1\|_{L^1}\sup_{[-a,a]}|u|  \sup_{[-a,a]}|u'| +  2h^{-2}(\| V_\infty\|_{L^\infty} + \|Pu\|_{L^2})\left(\int_{-a}^a|u'|^2\right)^{1/2}. 
\end{gathered}
\end{equation*}
This is a system of inequalities of the form $x^2 \le 2yz+Ay^2 + B$, $y^2 \le C + Dx$, $z^2 \le E + Fyz + Gx$. After using the second to eliminate $y$, we obtain a system in $x$ and $z$ with quadratic left hand sides and subquadratic right hand sides. Hence $x$, $y$, and $z$ are each bounded in terms of $A, B, \dots, G$. Letting $a \to \infty$, we conclude that $u' \in L^2$, $u \in L^\infty$, and $u' \in L^\infty$. Hence  $\mathcal D_\textrm{max} = \mathcal D$.

Equip $P$  with the domain $\mathcal D_\textrm{max} = \mathcal D \subset L^2$. By integration by parts, $P \subset P^*$. But, by Sturm--Liouville theory,  $P^* \subset P$: see \cite[Lemma 10.3.1]{zet}. Hence $P=P^*$.\\
\end{proof}

\end{document}